\documentclass[reqno,11pt, centertags]{article}
\usepackage{amsmath,amsthm,amscd,amssymb,latexsym,upref}
\date{\today}

\input epsf
\usepackage{epsfig}
\usepackage[T2A,OT1]{fontenc}
\usepackage[ot2enc]{inputenc}

\usepackage[russian,english]{babel}

\usepackage{epic,eepicemu}

\usepackage{amsmath,amsthm,amscd,amssymb,latexsym,upref}
\date{\today}

\input epsf
\usepackage{epsfig}
\usepackage[T2A,OT1]{fontenc}
\usepackage[ot2enc]{inputenc}
\usepackage{babel}

\newcommand{\bbH}{{\mathbb{H}}}

\newcommand{\bbR}{{\mathbb{R}}}

\newcommand{\bbZ}{{\mathbb{Z}}}
\newcommand{\bbC}{{\mathbb{C}}}

\newcommand{\bbT}{{\mathbb{T}}}

\newcommand{\cA}{{\mathcal{A}}}

\newcommand{\cR}{{\mathcal{R}}}

\newcommand{\cB}{{\mathcal{B}}}
\newcommand{\cC}{{\mathcal{C}}}

\renewcommand{\k}{\varkappa}

\renewcommand{\l}{\ell}
\renewcommand{\Re}{\text{\rm Re\,}}
\renewcommand{\Im}{\text{\rm Im\,}}

\allowdisplaybreaks \numberwithin{equation}{section}
\newtheorem{theorem}{Theorem}[section]

\newtheorem{lemma}[theorem]{Lemma}

\theoremstyle{definition}

\newtheorem{remark}[theorem]{Remark}
\newtheorem{problem}[theorem]{Problem}

\title
{Parametrization of spectral surfaces  of a  class of periodic  5-diagonal matrices}
\author{Ionela Moale and Peter Yuditskii\footnote{This work was supported by Austrian Science Fund FWF, project no: P22025-N18.}}
\date{\today}

\begin{document}

\maketitle

\begin{abstract}
The main result of this work is a parametric description of the spectral surfaces of a class of periodic 5-diagonal matrices, related to the strong moment problem. This class is a self-adjoint twin of the class of CMV matrices. Jointly they form the simplest possible classes of 5-diagonal matrices. 
\end{abstract}

  

\section{Introduction}
The main result of this work is a parametric description of the spectral surfaces of a class of periodic 5-diagonal matrices. This class is a self-adjoint twin of the quite popular now class of CMV matrices \cite{Sim, Sim2}. Jointly they form the simplest possible classes of 5-diagonal matrices, we explain this claim in details in Section \ref{s2}. We take a risk to introduce an abbreviation  and call them SMP (Strong Moment Problem) matrices, see   e.g. \cite{ON} and the remark below.

In a generic case SMP matrix is defined as a ratio of two 3-diagonal matrices
$A=A_2A_1^{-1}$, where
\begin{equation*}\label{tone}
A_1=\begin{bmatrix}
\ddots& \\
\ddots&\bar\pi_{-1}&  \\
\ddots&\sigma_{-1}& 0  \\
&\pi_0&1&\bar  \pi_1 \\
& & 0&\sigma_1&0 \\
& & & \pi_2&1& \ddots \\
& & & & 0&\ddots\\
& & & & & \ddots
\end{bmatrix},
\
A_2=\begin{bmatrix}
\ddots& \\
\ddots&0&  \\
\ddots&1& \bar p_0  \\
&0&q_0 & 0 \\
& & p_1&1&\bar p_2 \\
& & & 0&q_2& \ddots \\
& & & & p_3&\ddots\\
& & & & & \ddots
\end{bmatrix},
\end{equation*}
$A_1$ and $A_2$ are invertible, equivalently $\sigma_{2n-1}\not=0$ and $q_{2n}\not=0$ for all $n$.  Due to the symmetry $A=A^*$, the generating coefficient sequences are subject to the restrictions
\begin{equation*}\label{restr}
\frac{p_{2n-1}}{q_{2n-2}}=-\frac{\pi_{2n-1}}{\sigma_{2n-1}}, \quad
\frac{p_{2n}}{q_{2n}}=-\frac{\pi_{2n}}{\sigma_{2n-1}}.
\end{equation*}

Note that $A_1^{-1}$ and $A_2^{-1}$ are also 3-diagonal, thus $A$ and $A^{-1}$ are 5-diagonal matrices represented as products of 3-diagonal matrices.
Note also that degenerations are possible, see Remark \ref{rem43}.

Let $\{e_n\}$ be the standard basis in the two-sided $\l^2$. According to the above definition
\begin{equation*}\label{direct}
\begin{split}
A e_{2n}=A_2 e_{2n}=\bar p_{2n}e_{2n-1}+&q_{2n}e_{2n} +p_{2n+1}e_{2n+1}\\
A^{-1} e_{2n-1}=A_1 e_{2n-1}=\bar \pi_{2n-1}e_{2n-2}+&\sigma_{2n-1}e_{2n-1} +\pi_{2n}e_{2n}
\end{split}
\end{equation*}
Thus, constructing $A$ we follow  the procedure, which is similar to the CMV matrices case:
having $e_{-1}$ and $e_0$ as the generators of the cyclic subspace we form the whole space applying $A$ on the even step and $A^{-1}$ on the odd step, however, as it was mentioned, the operator is not unitary, but self-adjoint, that is, the spectrum is not on the unit circle but on the real axis. From this point of view the situation is similar to the orthogonalization procedure in the strong moment problem construction \cite{ON}, however because of periodicity  we are interested in two-sided matrices; it  is very essential: we do not assume that  the operator $A$ is positive! 

 $A$ can be represented as a \textit{two dimensional perturbation} of a block orthogonal matrix
$$
A=\begin{bmatrix}
A_-&0\\0&A_+
\end{bmatrix}+e_{-1}\langle\cdot, \tilde e_0\rangle \tilde p_0+
\tilde e_{0}\langle\cdot,  e_{-1}\rangle \tilde p_0,
$$
where 
$$
\tilde p_0=\|P_+ Ae_{-1}\|, \quad \tilde e_0=\frac 1 {\tilde p_0} {P_+ A e_{-1}},
$$
$A_\pm=P_\pm A P_\pm$ are restrictions of $A$ to the positive and negative half-axis according to the orthogonal decomposition $\l^2=\l^2_-\oplus \l_+^2$. For this reason certain general facts from its spectral theory can be reduced to the spectral theory of Jacobi matrices.  However it is quite different as soon as we pose the problem:
\begin{problem}
Describe the spectral sets of periodic SMP matrices. 
\end{problem}

Recall that for periodic Jacobi matrices the  spectrum is a system of intervals, which possesses the following parametric description, see \cite{EY} and references therein.  For a system of nonnegative parameters  $\{h_k\}_{k=1}^{n-1}$, let $D=D(h_1,\dots h_{n-1})$ be the region obtained from the half-strip
$$
\{w: -\pi n<\Re w<0, \ \Im w>0\} 
$$
by removing vertical intervals 
$$
\{ w: \Re w= -\pi k,\ 0<\Im w\le h_k \}, \quad k=1,\dots,n-1.
$$
Let $\theta$ be the conformal map of the upper half-plane $\bbH$ to $D$ normalized by the conditions $\theta(a_0)=0$, $\theta(b_0)=-\pi n$, $\theta (\infty)=\infty$. Denote by 
$E(h_1,\dots,h_{n-1})$ the full preimage of the interval $[-\pi n,0]\subset\partial D$, i.e.:
$$
E(h_1,\dots,h_{n-1}):=\theta^{-1}([-\pi n,0]).
$$
A system of intervals $E=[b_0,a_0]\setminus \cup _{j\ge 1}(a_j,b_j)$ is the spectrum of a periodic Jacobi matrix if and only if $E=E(h_1,...,h_{n-1})$ for a certain system of parameters 
$\{h_k\}_{k=1}^{n-1}$.

The spectral sets for periodic CMV matrices are given by conformal mappings onto similar \textit{periodic comb-like domains}.
To formulate our main result we define comb regions of a new kind.

 \begin{figure}
\begin{center} \includegraphics[scale=1]
{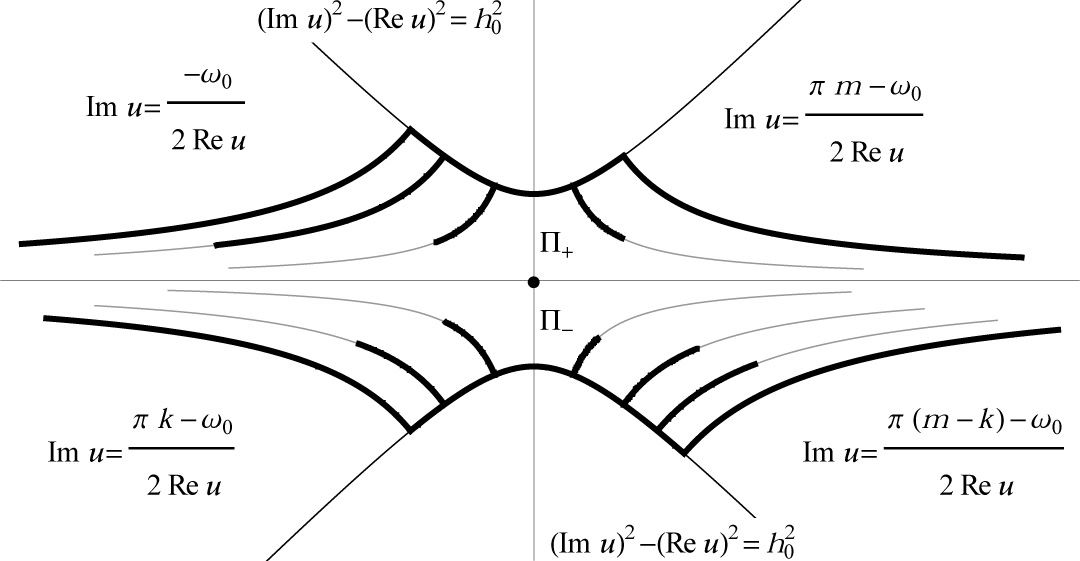}\end{center}
\caption{$\Pi$ region for $\omega_0\not=\pi \l$}\label{fig1}
\end{figure}

 \begin{figure}
\begin{center} \includegraphics[scale=1]
{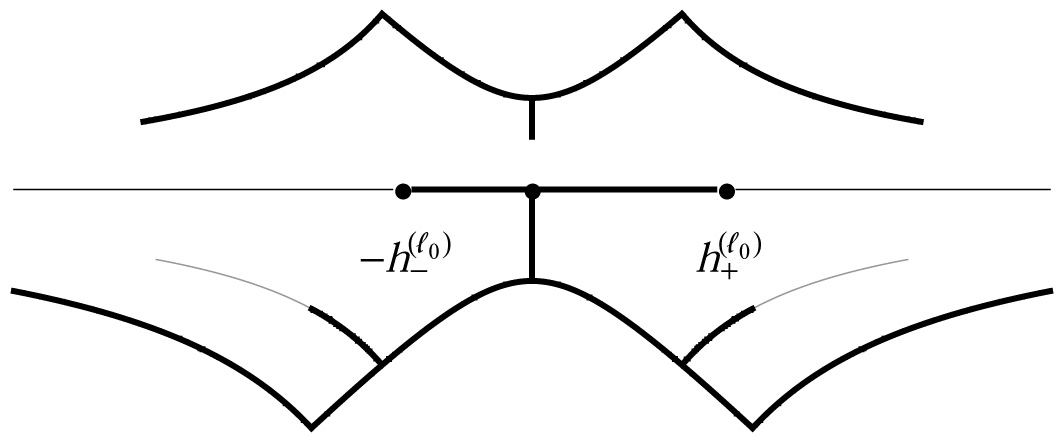}\end{center}
\caption{$\Pi$ region for $\omega_0=\pi \l_0$.}\label{fig2}
\end{figure}

For integers $k$ and $m\in[0,k]$ and parameters $h_0>0$ and $\omega_0$, 
$0\le\omega_0\le\pi m$ consider the region $\Pi_{k}^m(h_0, \omega_0)$ bounded by the hyperbolic curves
\begin{equation}\label{boundud}
 (\Im u)^2< (\Re u)^2+h^2_0,
\end{equation}
and the orthogonal systems of hyperbolas
\begin{equation}\label{boundlr}
\begin{split}
\Im u<\frac{\pi m-\omega_0}{2\Re u}, \quad  &\Im u>\frac{\pi (m-k)-\omega_0}{2\Re u},
\quad\text{for}\ \Re u>0,\\
\Im u<\frac{-\omega_0}{2\Re u}, \quad  &\Im u>\frac{\pi (k)-\omega_0}{2\Re u},
\quad\text{for}\ \Re u<0.
\end{split}
\end{equation}
If $\omega_0\not=\pi \l$, $0\le \l\le m$, by $\Pi$ we denote the region which is obtained from 
$\Pi_{k}^m(h_0, \omega_0)$ by removing pieces of hyperbolic curves
\begin{equation}\label{cut1}
\Im u=\frac{\pi\l-\omega_0}{2\Re u}, 
\quad\text{for}\ \Im u>0, \ 1\le \l\le m-1
\end{equation}
and
\begin{equation}\label{cut2}
\Im u=\frac{\pi(\l-k)-\omega_0}{2\Re u}, 
\quad\text{for}\ \Im u<0,\ m+1\le \l\le 2k-1
\end{equation}
of length $h_\l$, $h_l\ge 0$, see Fig. \ref{fig1}. If $\omega_0=\pi \l_0$ then the hyperbolic curves related to $\l=\l_0$ in \eqref{cut1}  and $\l=k+\l_0$ in \eqref{cut2} degenerate. In this case the corresponding cuts are pieces of the imaginary axis, as soon as 
\begin{equation}\label{cut3}
h_{\l_0}<h_0 \text{ and } h_{k+\l_0}<h_0.
\end{equation}
Otherwise one of them still satisfies \eqref{cut3}, and another one has T-shape, see Fig. \ref{fig2}, consisting of the piece of the imaginary axis 
\begin{equation}\label{cut4}
0\le \Im u\le h_0 \text{ or  }-h_0\le \Im u\le 0,
\end{equation}
respectively, and of the real interval 
\begin{equation}\label{cut5}
-h^{(\l_0)}_- \le\Re u\le h^{(\l_0)}_+,\quad  h^{(\l_0)}_\pm\ge 0.
\end{equation}

\begin{theorem}
For a region $\Pi$ described by the conditions \eqref{boundud}--\eqref{cut5}, 
let $\theta:\bbH\to \Pi$ be a  conformal map such that $0$ and $\infty$ correspond to the infinite points in $\Pi$. 
A system of intervals $E$, $0,\infty\not\in E$, is the spectral set of a periodic matrix of SMP class if and only if  it corresponds to the preimage of the part of the boundary given by the condition \eqref{boundud}  for a certain $\Pi$.
\end{theorem}

The structure of the work is as follows:
the simplest possible spectral surfaces of periodic 5-diagonal matrices are discussed in Section 2. 
In Section 3 we  prove our  main theorem. 
The functional model for periodic SMP matrices is given in Section 4.

\bigskip

The second author is thankful to the organizers of the program
 \textit{Hilbert spaces of entire functions and spectral theory of self-adjoint differential operators}, at CRM, Barcelona, 2011,  and to Alex Eremenko for 
stimulating discussions. In a sense this paper is an addition to their joint work \cite{EY}.

\section{Spectral surfaces  with the maximal number of  boundary ovals
}\label{s2}

Let $J$ be a 5-diagonal self-adjoint matrix of period $d$
\begin{equation}\label{1}
J=rS^2+pS+q+S^{-1}\bar p+ S^{-2} r,
\end{equation}
where $S$ is the shift operator and $p,r,q$ are diagonal matrices of period $d$, such that $r_m>0$.
We recall certain fundamental facts from the spectral theory of multi-diagonal periodic matrices 
\cite{MvM}, adopting to the 5-diagonal case.

For 
\begin{equation}\label{m1}
j(w)=\begin{bmatrix}
q_0&\bar p_1 & r_2        &\dots&0&r_{0}/w&p_0/w\\
p_1& q_1& \bar p_2&r_3& \dots&0           &r_{1}/w\\
r_2&p_2& q_2& \bar p_3&r_4& \dots&0           \\
 \dots& \dots& \dots&  \dots& \dots& \dots&\dots        \\
  0& \dots& r_{d-3}&  p_{d-3}& q_{d-3}& \bar  p_{d-2}   &  r_{d-1}   \\
r_{0}w&0&\dots &r_{d-2}& p_{d-2}& q_{d-2}&\bar p_{d-1}\\
\bar p_{0}w&r_1w &0&\dots & r_{d-1}&  p_{d-1}&q_{d-1}\\
\end{bmatrix}
\end{equation}
let
\begin{equation*}
F(z,w)=\frac{\det \{j(w)-z\cdot I\}}{\prod_{j=0}^{d-1} r_j}
=w^2+1/w^2+A(z)w+A_*(z) 1/w+B(z)
\end{equation*}
where $A$ and $B$ are polynomials, in particular, for even $d=2k$
\begin{equation}\label{zkw}
\begin{split}
B(z)=&\frac{z^{2k}}{\prod_{j=0}^{2k-1} r_j}+\dots,\\
 A_*(z):=\overline{A(\bar z)}
=& \left(\frac{-1}{\prod_{j=0}^{k-1} r_{2j}}+\frac{-1}{\prod_{j=0}^{k-1} r_{2j+1}}\right) z^k+\dots
\end{split}
\end{equation}
Then the spectral curve corresponding to $J$  is of the form
\begin{equation}\label{2}
\cR=\{P=(z,w):F(z,w)=0\}.
\end{equation}
$\cR$ is endowed with an 
antiholomorphic involution  $\tau P:=(\bar z, 1/\bar w)$ for which
$$
\cR\setminus \partial \cR_+= \cR_+\cup\cR_-, \quad 
\cR_+=\{P=(z,w)\in \cR: |w|<1\},
$$
see Fig. \ref{fig3}.
\begin{figure}
\begin{center} \includegraphics[scale=0.8]
{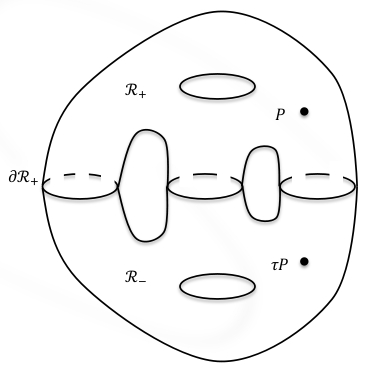}\end{center}
\caption{Topology of the spectral curve}\label{fig3}
\end{figure}
Note that the spectrum of $J$ (as the operator acting in $\l^2$) corresponds to the fixed line of the 
 involution $\tau$,
 $\tau P= P$,
$$
z\in\sigma(J)\Leftrightarrow \exists w: P=(z,w)\in\partial \cR_+.
$$
In other words it is described by the condition $|w|=1$.

Recall that the spectral surfaces related to periodic Jacobi matrices are of the form
$$
\tilde\cR=\left\{(w,z):\ w+\frac 1 w=\tilde A(z)\right\},
$$
where $\tilde A$ is a real polynomial. 
In the similar decomposition $\tilde\cR\setminus\partial \tilde\cR =\tilde\cR_+\cup\tilde\cR_-$
it possesses the following property: the number of boundary ovals, i.e., the number of intervals
$$
\partial \tilde\cR =\{z\in\bbR: |\tilde A(z)|\le 2\}
$$
  \textit{is maximal for the given genus of the surface}.

We say that the spectral curve $\cR$, related to a 5-diagonal matrix, is of the simplest structure if it has
 \textit{maximal possible number of components} of the boundary $\partial\cR_+$  for the given genus.
 For example, in Fig. \ref{fig3} the number of boundary components is 3, but its genus is 4 and the maximal possible number of components is 5. That is, the curve of this structure does not belong to the class.
 
 \begin{figure}
\begin{center} \includegraphics[scale=0.8]
{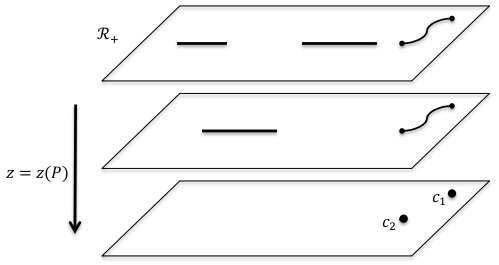}\end{center}
\caption{$\cR_+$ as the two sheeted covering of $z$-plane}\label{fig4}
\end{figure}
 In other words, let us represent $\cR_+$ as two a sheeted covering of the $z$-plane, see Fig. \ref{fig4}. It is a hyperelliptic curve with a system of cuts $\partial \cR_+$. We say that the spectral curve is of the simplest structure if this hyperelliptic curve has genus $0$, i.e.:
 \begin{equation*}\label{simeq}
  \cR_+\simeq \bar \bbC\setminus E.
\end{equation*}
The corresponding equivalence can be written explicitly
\begin{equation}\label{paramm}
z=\lambda+\frac{c^2}{\lambda-\lambda_0},\ \lambda\in \bbC,\quad c_1=\lambda_0-2c, \ c_2=\lambda_0+2c,
\end{equation}
where $c_1$, $c_2$ denote the only two possible critical values of $z$, in the case that both numbers are finite
and if, say, $c_2=\infty$ then
 $$
z=\lambda^2-2c,\quad c_1=c^2-2c.
$$
The set $E$, which corresponds to $\partial\cR_+$, is a system of cuts in the complex plane $\bbC$ with the property 
\begin{equation}\label{cuts0}
z(\lambda)\in\bbR \text{  for  }\lambda\in E.
\end{equation}

It is essential to note that  $E$ is far from being an arbitrary system of cuts for which 
\eqref{cuts0} holds. Recall that up to now the second function $w$ was not involved into considerations. Meanwhile $w=w(\lambda)$ is a function in 
$\bar\bbC\setminus E$ with the following properties \cite{MvM}:
\begin{itemize}
\item[(i)] $w$ is single-valued and holomorphic,
\item[(ii)] $|w|<1$ in $\bar\bbC\setminus E$ and $|w|=1$ on $E$,
\item[(iii)] zeros of $w$ are $\{\lambda_0,\infty\}=z^{-1}(\infty)\not\subset E$ (of equal multiplicity).
\end{itemize}
For definiteness, here and below, we consider the case \eqref{paramm} (with two finite critical values).
The properties (i)-(iii) imply that
\begin{equation}\label{wgreen}
\frac 1 k\log\frac{1}{|w(\lambda)|}= G_{\lambda_0}(\lambda)+G_\infty(\lambda),
\end{equation}
where $G_{\lambda_0}(\lambda)$ is the Green function in the domain $\bar \bbC\setminus E$ with a logarithmic pole at $\lambda_0$ and $k$ is the  multiplicity of $w$ in $\lambda_0$ and $\infty$ respectively. 

Let us recall the concept of the complex Green function $b_{\lambda_0}(\lambda)$. It is an analytic multivalued function in $\bar \bbC\setminus E$ such that
\begin{equation}\label{greengreen}
\log\frac{1}{|b_{\lambda_0}(\lambda)|}= G_{\lambda_0}(\lambda).
\end{equation}
Note that \eqref{greengreen} determines $b_{\lambda_0}$ up to a unimodular constant. In what follows we assume the normalizations $b_{\lambda_0}(\infty)>0$ and $b_\infty(\lambda_0)>0$.

Let $\pi_1(\bar\bbC\setminus E)$ be the fundamental group of this domain. Then $b_{\lambda_0}$
generates the character $\mu_{\lambda_0}\in \pi_1(\bar\bbC\setminus E)^*$ on this group by
$$
b_{\lambda_0}\circ\gamma=\mu_{\lambda_0}(\gamma) b_{\lambda_0},
\quad \gamma\in \pi_1(\bar\bbC\setminus E),
$$
which indicates the multivalued structure of the complex Green function. Moreover, let us split $E$
into connected components, $E=\cup_{j=0}^{\k} E_j$, and let $\gamma_j$'s be simple contours around
$E_j$'s. Note that they form generators of the group $\pi_1(\bar\bbC\setminus E)$ subject to the condition
$$
\gamma_0\circ\dots\circ \gamma_m= \text{trivial}.
$$
Then, for a suitable  choice of the direction of $\gamma_j$,
\begin{equation}\label{expharm}
\mu_{\lambda_0}(\gamma_j)=e^{2\pi i\omega_{\lambda_0}(E_j)},
\end{equation}
where $\omega_{\lambda_0}(E_j)$ is the harmonic measure of $E_j$ at $\lambda_0$.

The factor $b_{\lambda_0}b_\infty$ removes the singularities of $z$ in $\bar \bbC\setminus E$.
 Let
$$
t_\infty=\frac{(b_\infty b_{\lambda_0}z)(\infty)}{|(b_\infty b_{\lambda_0}z)(\infty)|}, \quad
t_{\lambda_0}=\frac{(b_\infty b_{\lambda_0}z)(\lambda_0)}{|(b_\infty b_{\lambda_0}z)(\lambda_0)|}, 
$$
Define $\xi\in[0,1)$ by the condition $t_\infty =e^{2\pi i \xi}t_{\lambda_0}$. Let $\mu=\mu_{\lambda_0}\mu_{\infty}$. Thus $\mu(\gamma_j)=e^{2\pi i \omega_j}$, where 
$\omega_j=\omega_{\lambda_0}(E_j)+\omega_\infty(E_j)$.

\begin{theorem}\label{harmcond}
Let $z$ be given by  \eqref{paramm}. Let $E=\cup_{j=0}^{\k} E_j\subset z^{-1}(\bbR)$ be a system of cuts \eqref{cuts0}. Then $\cR=\cR_+\cup\partial\cR_+\cup\cR_-$, where $\cR_+\simeq\bar \bbC\setminus E$, is a spectral curve of a 5-diagonal periodic matrix if and only if the numbers $\xi$ and $\omega_{j}$ (for all $j$) are rational. Moreover 
 $w=t_\infty^{-k}(b_{\lambda_0}b_\infty)^k$, where $k$ is a common denominator of these rational numbers.
\end{theorem}

A proof is based on a fact of the general theory \cite{MvM}, that a given periodic $J$ with the spectral curve $\cR$ possesses functional representation as the multiplication operator by $z$. We give such a representation following basically \cite{PY, SY}. 

For a fixed character $\alpha$ the multivalued analytic functions $F$, $F\circ\gamma=\alpha(\gamma) F$, such that $|F(\lambda)|^2$ has a harmonic majorant in $\bar\bbC\setminus E$, form the Hardy space $H^2(\alpha)\subset L^2_{\omega_\infty}$ with the  norm given by the integral of the boundary values:
$$
\|F\|^2=\int_E |F(\lambda)|^2\omega_\infty(d\lambda).
$$
Note that the point-evaluation functional is bounded in this space and therefore in $H^2(\alpha)$ there is the reproducing kernel $k^{\alpha}_\lambda$:
$$
F(\lambda)=\langle F, k^\alpha_\lambda\rangle, \quad \lambda\in\bar\bbC\setminus E,
$$
for all $F\in H^2(\alpha)$.

\begin{lemma}
For a character $\alpha\in \pi_1(\bar\bbC\setminus E)^*$ let 
$$
K_{\lambda}^\alpha=\frac{k_{\lambda}^\alpha}{\|k_{\lambda}^\alpha\|}
$$ 
denote the normalized reproducing kernel at $\lambda$.  Then, for an arbitrary system of unimodular constants $t_m$, the family
\begin{equation}\label{basis}
e_{2n}=t_{2n}b^n_{\lambda_0}b^n_\infty{K_{\lambda_0}^{\alpha\mu_{\lambda_0}^{-n}\mu^{-n}_\infty}}, \quad
e_{2n+1}=t_{2n+1}b^{n+1}_{\lambda_0} b^n_\infty {K_\infty^{\alpha\mu_{\lambda_0}^{-n-1}\mu^{-n}_\infty}}
\end{equation}
forms an orthonormal basis in $H^2(\alpha)$, $n\ge 0$. Moreover, extended to negative indexes it forms an orthonormal  basis in $L^2_{\omega_\infty}$.
\end{lemma}

\begin{proof} The system is orthonormal. Every function from $H^2(\alpha)$ orthogonal to it  has a zero of infinite multiplicity in $\lambda_0$ (and $\infty$) and therefore vanishes identically. To prove the second claim one has to use the description of the orthogonal complement $L^2_{\omega_\infty}\ominus H^2(\alpha)$ by means of the Hardy space \cite{SY} and once again apply the same argument related to the corresponding $H^2$-space and orthonormal basis of reproducing kernels in it.

\end{proof}

\begin{lemma}\label{lmodel}
The multiplication by $z$ with respect to the basis \eqref{basis} is a 5-diagonal self-adjoint matrix,
\begin{equation}\label{ldecomp}
ze_{m+2}=\bar r_{m+2}e_m+\bar p_{m+2}e_{m+1}+q_{m+2} e_{m+2}+
p_{m+3}e_{m+3}+r_{m+4}e_{m+4}.
\end{equation}
 Moreover  $r_m>0$ if and only if 
 \begin{equation}\label{sign}
t_{2n-1}=t_{\infty}^{-n}t_{(-1)}, \quad t_{2n}=t_{\lambda_0}^{-n}t_{(0)}.
\end{equation}

\end{lemma}

\begin{proof} Since the factor $b_{\lambda_0}b_\infty$ removes the singularities of $z$ in $\bar \bbC\setminus E$ the function $z b_{\lambda_0}b_\infty  F$ belongs to $H^2(\alpha\mu)$ for every $F\in H^2(\alpha)$ and an arbitrary $\alpha\in \pi_1(\bar \bbC\setminus E)^*$. Thus the decomposition of $ze_{m+2}$ starts with $e_m$,
\begin{equation*}\label{ldecomp1}
ze_{m+2}=\bar r_{m+2}e_m+\dots
\end{equation*}
Since $z(\lambda)$ is real on $E$ the multiplication operator is self-adjoint; its matrix possesses the symmetry property and therefore it is 5-diagonal \eqref{ldecomp}. Finally we put $\lambda=\lambda_0$ in \eqref{ldecomp} for even $m$
$$
(zb_{\lambda_0}b_\infty)(\lambda_0){K_{\lambda_0}^{\alpha\mu^{-(n+2)}}}(\lambda_0)t_{2n+2}=
\bar r_{2n+2} {K_{\lambda_0}^{\alpha\mu^{-n}}}(\lambda_0)t_{2n}
$$
and $\lambda=\infty$  for odd $m$
$$
(zb_{\lambda_0}b_\infty)(\infty){K_{\infty}^{\alpha\mu_{\infty}\mu^{-(n+2)}}}(\infty)t_{2n+1}=
\bar r_{2n+1} {K_{\infty}^{\alpha\mu_{\infty}\mu^{-n}}}(\infty)t_{2n-1}.
$$
Since $K^{\alpha}_\lambda(\lambda)>0$ we get \eqref{sign}.
\end{proof}

\begin{proof}[Proof of Theorem \ref{harmcond}] Let $J$ be a periodic self-adjoint 5-diagonal matrix and $\cR$ be its spectral surface such that $\cR_+\simeq \bar \bbC\setminus E$. Since $w(\lambda)$ is single-valued in the domain, \eqref{wgreen} and \eqref{expharm} imply that $\omega_j=\omega_{\lambda_0}(E_j)+\omega_\infty(E_j)$ are rational. Futher, due to \eqref{zkw} the function $w z^k$ is regular in the domain, moreover $1/(w z^k)(\lambda_0)$ and
 $1/(w z^k)(\infty)$ are roots of the quadratic equation
 \begin{equation*}
 \begin{split}
 C^2+\left(\frac{-1}{\prod_{j=0}^{k-1}r_{2j}}+\frac{-1}{\prod_{j=0}^{k-1}r_{2j+1}}\right) C+
 \frac{1}{\prod_{j=0}^{2k-1}r_{j}}\\=
\left(C-\frac{1}{\prod_{j=0}^{k-1}r_{2j}}\right)\left(C-\frac{1}{\prod_{j=0}^{k-1}r_{2j}}\right)=0.
 \end{split}
\end{equation*}
Thus $(w z^k)(\lambda_0)>0$ and  $(w z^k)(\infty)>0$. Since $w=t (b_{\lambda_0}b_\infty)^k$, 
$t\in \bbT$, the ratio $(zb_{\lambda_0}b_\infty)^k(\infty)/(zb_{\lambda_0}b_\infty)^k(\lambda_0)$
is also positive. That is $e^{2\pi i k\xi}=1$. And this finishes the necessity part of the theorem.

In the opposite direction, for the given system of cuts we define $J$ according to Lemma \ref{lmodel}. It remains to check that $J$ is periodic. Let $w=t_\infty^{-k}(b_{\lambda_0}b_\infty)^k$. 
Since $\mu^k(\gamma)=1$, for all $\gamma\in\pi_1(\bar\bbC\setminus E)$,  it is single-valued.
Note that $w$ is normalized by the condition
$(wz^k)(\infty)=|(b_{\lambda_0}b_\infty z)^k(\infty)|>0$. We claim that
\begin{equation}\label{mshift}
w e_n=e_{n+2k}.
\end{equation}
For odd $n$ \eqref{mshift} holds  automatically. For even $n$ we should take into account that in addition $t^{-k}_{\infty}t_{\lambda_0}^k=
e^{-2\pi i k \xi}=1$. Thus \eqref{mshift} defines the shift operator. Since the multiplication operators by $z$ and $w$ commute, we have
$JS^{2k}=S^{2k} J$. Therefore $J$ is periodic.
\end{proof}
 
 Now, let us restrict ourselves to the \textit{real} case, i.e., $c_2=\bar c_1$ or both critical values are real. Without lost of generality  $c_2=\bar c_1=2i$ or $c_2=-c_1=2$. Thus, according to \eqref{paramm},
 \begin{equation}\label{paramm1}
z=\lambda-\frac 1\lambda
\end{equation}
in the first case, and
\begin{equation}\label{paramm2}
z=\lambda+\frac 1\lambda
\end{equation}
in the second one.

In the case \eqref{paramm1}, $z^{-1}(\bbR)= \bbR$, thus $E$ is a system of intervals on the real axis. Since $z^{-1}(\infty)= \{0,\infty\}$, $E$ is subject to the restriction  $ \{0,\infty\}\not\subset E$.

If $z$ is of the form \eqref{paramm2}, then $z^{-1}(\bbR)= \bbR\cup\bbT$, that is, $E$ is a union of real intervals and arcs of the unit circle, and again $\{0,\infty\}\not\subset E$. This case under the additional assumption $E\subset\bbT$ leads to periodic CMV matrices \cite{Sim2}.
Indeed,  in the current case the multiplication by $\lambda$ is also well defined and represents a unitary matrix $A$ such that 
\begin{equation}\label{fromCMV}
J=A+ A^{-1}=A+A^*.
\end{equation}
 As it was mentioned this functional model is the same as that related to periodic and almost periodic CMV
matrices, see e.g. \cite{PY}.
Conversely, having a periodic CMV matrix $A$ we obtain the periodic self-adjoint $J$ of the class by 
\eqref{fromCMV}.

Similarly, in the case \eqref{paramm1} the multiplication by $\lambda$ leads to the self-adjoint operator $A$ such that
\begin{equation}\label{fromSMP}
J=A- A^{-1},
\end{equation}
where $A^{-1}$ exists and  corresponds to the multiplication by $1/\lambda$, i.e., to a periodic SMP matrix. Further details of the corresponding functional model are discussed in Section \ref{sect4}.
Note that the case \eqref{paramm2} under the additional assumption $E\subset \bbR$ leads to essentially the same class of self-adjoint operators.

\section{Proof of the Main Theorem}
Let  $E=[b_0,a_0]\setminus \cup_{j= 1}^\k(a_j,b_j)$ be a system of intervals on $\bbR$, recall $\{0,\infty\}\in \bar\bbC\setminus E$, $z=\lambda-1/\lambda$. We apply Theorem \ref{harmcond} in the current case. As it is well known the Green function (say with respect to infinity) is represented by the hyperelliptic integral, see e.g. \cite{Tom, SY},
 $$
 G(\lambda,\infty)=\Re \int^\lambda_{a_0}\frac{\lambda^{\k}+\dots}{\sqrt{\prod_{j=0}^\k(\lambda-a_j)(\lambda-b_j)}}d\lambda.
 $$
 Therefore for the sum of the Green functions we have
 \begin{equation}\label{gfinl}
 G(\lambda,\infty)+G(\lambda,0)=\Re \int^\lambda_{a_0}\frac{M_{\k+1}(\lambda)}{\sqrt{\prod_{j=0}^\k(\lambda-a_j)(\lambda-b_j)}}\frac{d\lambda}\lambda,
\end{equation}
where $M_{\k+1}$ is a monic polynomial of degree $\k+1$. Note that the residue of the corresponding differential at  the origin is $-1$.
 
 \subsection{Spectrum of SMP matrices for the Stieltjes class}\label{sb31}
 Let us consider the simplest case $E\subset \bbR_+$ or $E\subset \bbR_-$
 (we can say that the spectrum is on the upper (lower) sheet of $\cR_+$). The strong Stieltjes moment problem is related to measures supported on the positive half-axis \cite{ON}. The shape of the sum $G(\lambda,\infty)+G(\lambda,0)$ on $\bbR\setminus E$ is shown in Fig. \ref{fig5}. It implies immediately that all the zeros of the polynomial $M_{\k+1}$ in \eqref{gfinl} are real in this case. Indeed, each gap $(a_j,b_j)$, $j\ge 1$, contains at least one critical point; there is a critical point between $-\infty$ and $0$; the total number of critical points is $\k+1$. Therefore
 \begin{equation}\label{abint}
\tilde \theta(\lambda)=i \int^{\lambda}_{a_0}\frac{M_{\k+1}(\lambda)}{\sqrt{\prod_{j=0}^\k(\lambda-a_j)(\lambda-b_j)}}\frac{d\lambda}\lambda
\end{equation}
is the Schwarz-Christoffel integral, which maps conformally the upper half-plane    $\bbH$ onto the (generalized) polygon in Fig. \ref{fig6}.

\begin{figure}
\begin{center} \includegraphics[scale=0.8]
{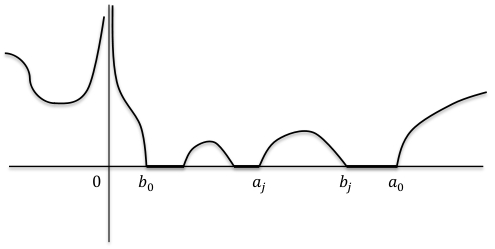}\end{center}
\caption{$G_0+G_\infty$ on $\bbR$: the spectrum is on the upper sheet}\label{fig5}
\end{figure}

\begin{figure}
\begin{center} \includegraphics[scale=1]
{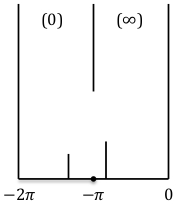}\end{center}
\caption{Image of the Abelian integral $\tilde \theta$}\label{fig6}
\end{figure}

According to \eqref{wgreen}   $w(\lambda)=te^{ik\tilde\theta(\lambda)}$, $t\in\bbT$. Let $\tilde\omega_j$ be the coordinates of the base of the slits. Then $w(\lambda)$ is single valued in $\bar\bbC\setminus E$ if and only if $\tilde\omega_j k\in \pi \bbZ$ for all $j$. It remains to mention that due to the chosen normalization for 
$b_{\lambda_0}$ and $b_\infty$ the product $b_{\lambda_0}(\lambda)b_{\infty}(\lambda)(\lambda-1/\lambda)$ is positive at infinity and negative at the origin, that is $\xi=1/2$. 
Thus we can parametrize the spectral sets of periodic SMP matrices in this case by sufficiently simple domains shown in Fig. \ref{fig6} with rational $\tilde\omega_j$'s (quite similar to the Jacobi and CMV matrices cases).

\subsection{Complex critical points and three real critical points in the same gap}
The situation changes dramatically as soon as  $0\in (a_j,b_j)$, $j\ge 1$, that is, $E=E_-\cup E_+$, 
$E_\pm\subset \bbR_\pm$. Still all gaps, except for $(a_j,b_j)$, should contain a critical point, thus $M_{\k+1}$ has at least $\k-1$ real critical points. However, the positions of two remaining critical points are not a priory fixed.

First, we consider the case when \textbf{two} remaining \textbf{critical points are complex} $\mu_0$ and $\overline{\mu_0}$, $\Im\mu_0 >0$. Let us consider $\tilde\theta(\lambda)$ in the upper half-plane $\bbH$. 
Since locally $\tilde \theta(\lambda)=\tilde \theta(\mu_0)+C(\lambda-\mu_0)^2+\dots$, 
there exist two orthogonal directions where   $\Re d\tilde \theta=0$. Moreover for one of them 
$\Im \tilde \theta$ has a local minimum at $\mu_0$ and a local maximum for another one.
We define the curve $\gamma$, $\mu_0\in\gamma$,  by the condition $\Re d\tilde \theta=0$, such that $\Im\tilde\theta$ increases.  Since there is no other critical point in $\bbH$ this curve should terminate on the real axis. Note that $\Im \tilde \theta(\lambda)$ decreases as  $\lambda$ approaches $E$.  If so, in the gaps $\gamma$ may approach either  a critical point  or  $0$ and $\infty$. The first case is also not possible  since the critical point is a local minimum for 
$\Im\tilde \theta=G_{\lambda_0}+G_\infty$ along the real axis, thus it should be local maximum in the orthogonal direction $\gamma$. But along $\gamma$ it increases. Thus, $\gamma$ terminates at $0$ and $\infty$,   see Fig. \ref{fig7}.

\begin{figure}
\begin{center} \includegraphics[scale=0.8]
{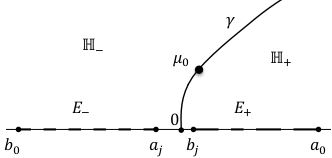}\end{center}
\caption{A complex critical point}\label{fig7}
\end{figure}

As the result we get $\bbH\setminus\gamma=\bbH_-\cup \bbH_+$ such that $E_{\pm}\subset\partial \bbH_{\pm}$. Let $\theta(\lambda)=k\tilde\theta(\lambda)$. Inspecting the boundary behavior of the given analytic function we obtain that it  maps conformally $\bbH_{\pm}$ onto $D_{\pm}$ shown in Fig. \ref{fig8}, where the  point $-\omega_0+ih^2_0$ corresponds to the critical point $\mu_0$.

 \begin{figure}
\begin{center} \includegraphics[scale=0.8]
{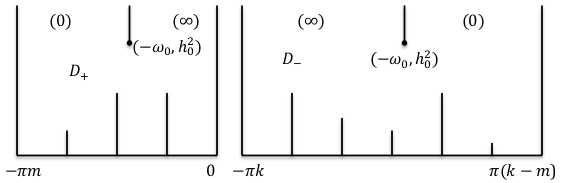}\end{center}
\caption{$D_\pm=\theta(\bbH_\pm)$ regions for a complex critical point}\label{fig8}
\end{figure}

Now we define
\begin{equation}\label{uofth}
u(\lambda)=\sqrt{-i\theta(\lambda)-i\omega_0-h^2_0}\quad\text{for}\ \lambda\in \bbH_+
\end{equation}
Here we assume that $\Im u(\lambda)>0$. Similarly we define
\begin{equation}\label{uofthminus}
u(\lambda)=-\sqrt{-i\theta(\lambda)-i\omega_0-h^2_0}\quad\text{for}\ \lambda\in \bbH_-
\end{equation}
 and in this case $\Im u(\lambda)<0$. In this way we get the regions $\Pi_\pm$.
 Since
 $$
 \theta=-(2\Re u \,\Im u+\omega_0)+i((\Re u)^2-(\Im u)^2+h_0^2),
 $$
these regions are   bounded by hyperbolic curves \eqref{boundud}--\eqref{cut2}, see Fig. \ref{fig1}.
 Gluing the images along the curve $\gamma$ we obtain the conformal mapping of the upper half-plane $\bbH$ onto the special comb domain $\Pi=\Pi_+\cup\Pi_-\cup \bbR$.
 
 Conversely, for the region $\Pi$ described by these equations  we define a conformal map 
 $u:\bbH\to\Pi$ , $u(0)=+\infty$, $u(\infty)=-\infty$, and set
 \begin{equation}\label{spc1b}
w(\lambda)=e^{-(u^2(\lambda)+h^2_0+i\omega_0)}, \quad z=\lambda-\frac 1{\lambda}
\end{equation}
Then, the set $E$ corresponds to $|w|=1$. Since the base of the slits for $D_\pm$ are of the form $\pi \l$, $w$  extended in the lower half-plane is single-valued in $\bar\bbC\setminus E$. Finally,
$w$ is real for $\lambda\in\bbR\setminus E$, that is, $\xi$ is rational. Based on Theorem 
\ref{harmcond} we conclude that this domain can be associated to a periodic SMP matrix.

\smallskip

 Let us turn to the case of \textbf{three real critical points in the same gap.}  In this case $\bbH$ can be decomposed into three pieces. Let $\mu_1<\mu_0<\mu_2$ be critical points in the gap 
 $(a_i,b_i)$. Note that with necessity $\mu_1$ and $\mu_2$ are points of local maximum and
  $\Im \theta$ assumes a local minimum at $\mu_0$ in this interval. Therefore there are directions 
  $\gamma_1$, $\gamma_2$ orthogonal to the real axis at $\mu_1$ and $\mu_2$ respectively such that $\Im\theta$ increases.
  Arguments like the above show that   these curves, $\Re d\theta=0$, terminate at $0$ and $\infty$, see Fig. \ref{fig9}
  
\begin{figure}
\begin{center} \includegraphics[scale=0.9]
{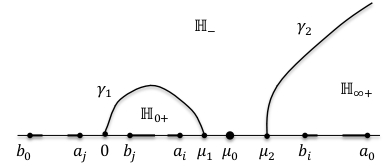}\end{center}
\caption{Three real critical points in the same gap,
$\bbH_{+}=\bbH_{0+}\cup\bbH_{\infty+}$}\label{fig9}
\end{figure}

In each of them, $\theta(\lambda)$ represents a conformal mapping, see Fig. \ref{fig10}. 
In this picture $-\omega_0+i h_0^2$, $-\omega_0+i h_1^2$, and $-\omega_0+ ih_2^2$ are images of the critical points $\mu_0$, $\mu_1$, and  $\mu_2$ respectively and $\omega_0=\pi \l_0$.
We make the change of variable \eqref{uofth}, \eqref{uofthminus}, having in mind that now $\bbH_+$ or $\bbH_-$ consists of two components. We arrive at the parametrization of the spectral curve by the domains of the form Fig. \ref{fig2} such that 
$$
h_-^{(\l_0)}=\sqrt{h_2^2-h_0^2}, \quad h^{(\l_0)}_{+}=\sqrt{h_1^2-h_0^2}
$$
in \eqref{cut5}.
 \begin{figure}
\begin{center} \includegraphics[scale=0.9]
{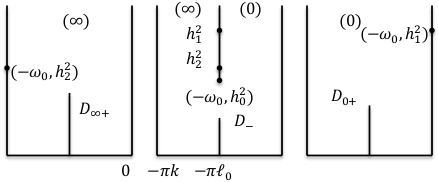}\end{center}
\caption{$D_\pm=\theta(\bbH_\pm)$ three real critical points in the same gap}\label{fig10}
\end{figure}

As before, starting from a region $\Pi$, by \eqref{spc1b} we arrive at the set $E$ and the domain
$\bar \bbC\setminus E\simeq\cR_+$ which corresponds to a periodic SMP matrix.

\subsection{Other cases} In the previous subsection we considered critical values in two main  generic positions. Now let us list the remaining special cases. 

1. For two complex critical values, if $\omega_0=\pi \l_0$ then one of the cuts   in \eqref{cut1} and one in \eqref{cut2} degenerates to the intervals on the imaginary axis. The length of  such a cut can not be arbitrary long, thus $h_{\l_0}$ and $h_{k+\l_0}$ are subject for the conditions \eqref{cut3}. As soon as one of these values approaches  $h_0$ two complex critical values, from the upper and lower half-planes,  approach  the critical value in the corresponding gap. In the limit we have the critical value of multiplicity 3. The same special case can be obtained when two critical values $\mu_1$ and $\mu_2$ tend to $\mu_0$, correspondingly $h_\pm^{(\l_0)}\to 0$, see Fig. \ref{fig2}.

2. The case of a critical point of multiplicity two and a simple critical point in a gap corresponds to   $h_+^{(\l_0)}= 0$, $h_-^{(\l_0)}> 0$  or $h_-^{(\l_0)}= 0$, $h_+^{(\l_0)}>0$.

3. Two critical points (or one critical point of multiplicity two) may appear in the  interval which contains zero or infinity. The domain $\Pi$ looks similar to that one shown in Fig. \ref{fig2}, but the degenerated hyperbola corresponds to the most left (or right) position, i.e.,  $\l_0=0$ or $\l_0=m$.

4. It was assumed that $m\le k$. If $m> k$ the domain $\Pi$ in Fig. \ref{fig1} remains the same, but we switch  the normalization conditions to $u(0)=-\infty$ and $u(\infty)=+\infty$.

5. In the Stieltjes case, subsection \ref{sb31}, the spectral curve was described by a simpler domain, Fig. \ref{fig6}. By \eqref{uofth} it can be transformed to a $\Pi$ region bounded from below by the real axis.

\section{Functional model for periodic SMP matrices}\label{sect4}

 Let $\bar \bbC\setminus E$ correspond to a periodic SMP matrix. We define
\begin{equation}\label{deff1}
\cA(\alpha)=\frac{K_0^\alpha(\infty)}{ K_\infty^{\alpha}(\infty)},\quad
\cB(\alpha)=\frac{K_0^\alpha(0)}{ K_\infty^{\alpha}(\infty)},\quad 
\alpha\in\pi_1(\bar\bbC\setminus E)^*.
\end{equation}
For the reader's convenience we prove here a  known lemma, see e.g. \cite{PY}.

\begin{lemma}The following identities hold true
\begin{equation}\label{deff3}
\cC(\alpha):=\sqrt{1-|\cA(\alpha)|^2}=b_\infty(0)\frac{K_0^{\alpha\mu_\infty^{-1}}(0)}{K_0^\alpha(0)}
\end{equation}
and
\begin{equation}\label{deff2}
\overline{\cA(\alpha)}=\frac{K_\infty^\alpha(0)}{ K_0^{\alpha}(0)},\quad
\cC(\alpha)=b_0(\infty)\frac{K_\infty^{\alpha\mu_0^{-1}}(\infty)}{K^\alpha_\infty(\infty)}.
\end{equation}
\end{lemma}
\begin{proof}
 $\cA(\alpha)$ is defined  by the following orthogonal  decompositions
\begin{equation*}
\begin{matrix}
K_0^\alpha&=&\cA(\alpha) K_\infty^{\alpha}&+&\sqrt{1-|\cA(\alpha)|^2} 
b_\infty K_0^{\alpha\mu_\infty^{-1}}\\
b_0 K_\infty^{\alpha\mu_0^{-1}}&=& \sqrt{1-|\cA(\alpha)|^2}K_\infty^{\alpha}&-&
\overline{\cA(\alpha)} b_\infty K_0^{\alpha\mu_\infty^{-1}}
\end{matrix}
\end{equation*}
Indeed,
$\cA(\alpha)=\frac{K_0^\alpha(\infty)}{ K_\infty^{\alpha}(\infty)}$ and 
we get \eqref{deff3}.

Since, in addition,
\begin{equation*}
\begin{matrix}
K_\infty^{\alpha}&=&\overline{\cA(\alpha)} K_0^\alpha&+&\sqrt{1-|\cA(\alpha)|^2} 
b_0 K_\infty^{\alpha\mu_0^{-1}}\\
b_\infty K_0^{\alpha\mu_\infty^{-1}}&=& \sqrt{1-|\cA(\alpha)|^2} K_0^\alpha&-&
{\cA(\alpha)} b_0 K_\infty^{\alpha\mu_0^{-1}}
\end{matrix}
\end{equation*}
we have \eqref{deff2}.

\end{proof}

In what follows without loss of generality we assume that $t_{(-1)}=1$ in \eqref{sign}. In the given case $\lambda_0=0$, so $t_0$ is the new notation for $t_{\lambda_0}$, and this is not the same as the initial $t_{(0)}$. Since
$$
(b_\infty b_0 z)(\infty)= (b_\infty \lambda)(\infty) b_0(\infty),\quad
(b_\infty b_0 z)(0)= -(b_0/ \lambda)(0) b_{\infty}(0),
$$
we have $t_\infty=\phi_\infty/|\phi_\infty|$ and $t_0=-\phi_0/|\phi_0|$,
where
$$
\phi_\infty=\left(\frac{b_\infty\lambda}{b_0}\right)(\infty), \quad
\phi_0=\left(\frac{b_0}{b_\infty\lambda}\right)(0).
$$
Also, recall that $t_\infty/t_0=e^{2\pi i\xi}$.

\begin{theorem}
The multiplication operator by $\lambda$ with respect to the basis \eqref{basis} is  a periodic  SMP matrix $A=A(\alpha,t_{(0)})$ with the following coefficients 
\begin{equation}\label{modelco}
\begin{split}
\bar p_{2n}=&\phi_\infty t_{(0)}e^{2\pi i\xi n} \cA(\alpha\mu^{-n})\cB^{-1}(\alpha\mu^{-n})\cC(\alpha\mu_\infty\mu^{-n})
\cB(\alpha\mu_\infty\mu^{-n})
\\
p_{2n+1}=&-\bar \phi_\infty t_{(0)} t_\infty e^{2\pi i\xi n}
\cC(\alpha\mu^{-n})\cB(\alpha\mu^{-n})
\cA(\alpha\mu_{\infty}\mu^{-n})\cB^{-1}(\alpha\mu_\infty\mu^{-n})
\\
q_{2n}=&-\phi_\infty{ \cA(\alpha\mu^{-n})}\cB^{-1}(\alpha\mu^{-n})
\overline{\cA(\alpha\mu_\infty\mu^{-n})}
\cB(\alpha\mu_\infty\mu^{-n})
\end{split}
\end{equation}
and
\begin{equation}\label{modelco2}
\begin{split}
 \pi_{2n+1}=&t_{(0)} t_{\infty} e^{2\pi i\xi n}\bar \phi_0 
  \cC(\alpha\mu^{-n})\cB(\alpha\mu^{-n})
 \cA(\alpha\mu_\infty\mu^{-(n+1)})\cB^{-1}(\alpha\mu_\infty\mu^{-(n+1)})
\\
\bar\pi_{2n+2}=&-t_{(0)}e^{2\pi i\xi(n+1)}\phi_0
\cA(\alpha\mu^{-n})\cB^{-1}(\alpha\mu^{-n})\cC(\alpha\mu_\infty\mu^{-(n+1)})
\cB(\alpha\mu_\infty\mu^{-(n+1)})\\
\sigma_{2n+1}=&-\phi_0{ \cA(\alpha\mu^{-n})}\cB^{-1}(\alpha\mu^{-n})
\overline{\cA(\alpha\mu_\infty\mu^{-(n+1)})}
\cB(\alpha\mu_\infty\mu^{-(n+1)})
\end{split}
\end{equation}

\end{theorem}

\begin{proof}
Note that $b_{\infty}\lambda k_0^\alpha\in H^2(\alpha\mu_\infty)$ and it is orthogonal to
the subspace $b_0 b_\infty^2H^2(\alpha\mu^{-1}_0\mu_\infty^{-2})$. Therefore, in fact, we have the three-terms recurrence relation,
$$
\lambda e_{2n}=\bar p_{2n} e_{2n-1}+ q_{2n} e_{2n}+ p_{2n+1} e_{2n+1}.
$$
Moreover
\begin{equation*}
\bar p_{2n}=t_{(0)}e^{2\pi i\xi n}(\lambda b_\infty)(\infty) \frac{K_0^{\alpha\mu^{-n}}(\infty)}{K_\infty^{\alpha\mu_\infty\mu^{-n}}(\infty)},
\end{equation*}
\begin{equation*}
\begin{split}
q_{2n}=&\langle\lambda K_0^{\alpha\mu^{-n}}, K_0^{\alpha\mu^{-n}}\rangle\\
=&
\langle\lambda K_0^{\alpha\mu^{-n}}-\frac{(\lambda b_\infty)(\infty)K_\infty^{\alpha\mu_\infty\mu^{-n}}}
{b_\infty K_\infty^{\alpha\mu_\infty\mu^{-n}}(\infty)}K_0^{\alpha\mu^{-n}}(\infty), 
K_0^{\alpha\mu^{-n}}\rangle
\\
=&
-\frac{(\lambda b_\infty)(\infty)}
{b_\infty(0)}
\frac{K_\infty^{\alpha\mu_\infty\mu^{-n}}(0)}
{ K_\infty^{\alpha\mu_\infty\mu^{-n}}(\infty)}
\frac{K_0^{\alpha\mu^{-n}}(\infty)}{K_0^{\alpha\mu^{-n}}(0)},
\end{split}
\end{equation*}
and
\begin{equation*}
\begin{split}
p_{2n+1}=&\langle\lambda t_{(0)} t_0^{-n}K_0^{\alpha\mu^{-n}}, 
t_\infty^{-n-1}b_0 K_\infty^{\alpha\mu_0^{-1}\mu^{-n}}\rangle
\\
=t_{(0)}e^{2\pi i\xi n} t_\infty &\langle K_0^{\alpha\mu^{-n}},
\lambda b_0 K_\infty^{\alpha\mu_0^{-1}\mu^{-n}}
-\frac{(\lambda b_0 b_\infty)(\infty)K_\infty^{\alpha\mu_0^{-1}\mu^{-n}}(\infty)K_\infty^{\alpha\mu_\infty\mu^{-n}}}
{b_\infty K_\infty^{\alpha\mu_\infty\mu^{-n}}(\infty)}
\rangle
\\
=&-t_{(0)}e^{2\pi i\xi n} t_\infty\overline{
\frac{(\lambda b_\infty)(\infty)}
{b_\infty(0)}
\frac{K_\infty^{\alpha\mu_\infty\mu^{-n}}(0)}
{ K_\infty^{\alpha\mu_\infty\mu^{-n}}(\infty)}
\frac{b_0(\infty)K_\infty^{\alpha\mu_0^{-1}\mu^{-n}}(\infty)}{K_0^{\alpha\mu^{-n}}(0)}}.
\end{split}
\end{equation*}

In its turn,
$$
\frac 1 {\lambda} e_{2n+1}=\bar \pi_{2n+1} e_{2n}+ \sigma_{2n+1} e_{2n+1}+ \pi_{2n+2} e_{2n+2},
$$
where
\begin{equation*}
\bar \pi_{2n+1}=\bar t_{(0)}\bar t_{\infty} e^{-2\pi i\xi n}\left(\frac{ b_0}{\lambda}
\right)(0) \frac{K_\infty^{\alpha\mu_0^{-1}\mu^{-n}}(0)}{K_0^{\alpha\mu^{-n}}(0)},
\end{equation*}
\begin{equation*}
\begin{split}
\sigma_{2n+1}=&\langle\frac{1}{\lambda} K_\infty^{\alpha\mu_0^{-1}\mu^{-n}}, 
K_\infty^{\alpha\mu_0^{-1}\mu^{-n}}\rangle\\
=&\langle\frac 1\lambda K_\infty^{\alpha\mu_0^{-1}\mu^{-n}}-
\left(\frac{ b_0}{\lambda}\right)(0) 
\frac{K_0^{\alpha\mu^{-n}}}
{b_0 K_0^{\alpha\mu^{-n}}(0)}K_\infty^{\alpha\mu_0^{-1}\mu^{-n}}(0),  
K_\infty^{\alpha\mu_0^{-1}\mu^{-n}}\rangle
\\
=&
-\left(\frac{ b_0}{\lambda}\right)(0) \frac 1{b_0(\infty)}
\frac{K_0^{\alpha\mu^{-n}}(\infty)}{K_0^{\alpha\mu^{-n}}(0)}
\frac{K_\infty^{\alpha\mu_0^{-1}\mu^{-n}}(0)}
{ K_\infty^{\alpha\mu_0^{-1}\mu^{-n}}(\infty)},
\end{split}
\end{equation*}
and
\begin{equation*}
\begin{split}
\pi_{2n+2}=&\langle\frac{1}{\lambda} t_{\infty}^{-n-1} K_\infty^{\alpha\mu_0^{-1}\mu^{-n}}, 
t_{(0)} t_0^{-n-1} b_\infty K_0^{\alpha\mu^{-1}\mu^{-n}}\rangle\\
=\bar t_{(0)}e^{-2\pi i\xi(n+1)}&\langle  K_\infty^{\alpha\mu_0^{-1}\mu^{-n}}
,\frac {b_{\infty}} \lambda K_0^{\alpha\mu^{-n-1}}
-\frac{K_0^{\alpha\mu^{-n}}\left(\frac{ b_0 b_\infty }{\lambda} K_0^{\alpha\mu^{-n-1}}
\right)(0)}
{b_0 K_0^{\alpha\mu^{-n}}(0)}
\rangle
\\
=-\bar t_{(0)}e^{-2\pi i\xi(n+1)}&\overline{
\left(\frac{ b_0 b_\infty}{\lambda}\right)(0) \frac 1{b_0(\infty)}
\frac{K_0^{\alpha\mu^{-n}}(\infty)}{K_0^{\alpha\mu^{-n}}(0)}
\frac{K_0^{\alpha\mu^{-n-1}}(0)}
{ K_\infty^{\alpha\mu_0^{-1}\mu^{-n}}(\infty)}}.
\end{split}
\end{equation*}

Now by making use of \eqref{deff1}, \eqref{deff2}, \eqref{deff3},
we obtain \eqref{modelco}
\begin{equation*}
\begin{split}
q_{2n}=&
-\phi_\infty
\overline{\cA(\alpha\mu_\infty\mu^{-n})}\cB(\alpha\mu_\infty\mu^{-n})\cA(\alpha\mu^{-n})
\cB^{-1}(\alpha\mu^{-n}),
\\
\bar p_{2n}=&
\phi_\infty t_{(0)}e^{2\pi i\xi n} b_\infty(0)
\frac{K_0^{\alpha\mu^{-n}}(\infty)}{K^{\alpha\mu_\infty\mu^{-n}}_\infty(\infty)}
\\
=&
\phi_\infty  t_{(0)}e^{2\pi i\xi n} 
\frac{K_0^{\alpha\mu^{-n}}(\infty)}{K^{\alpha\mu^{-n}}_\infty(\infty)}
\frac{K_\infty^{\alpha\mu^{-n}}(\infty)}{K_\infty^{\alpha\mu_\infty\mu^{-n}}(\infty)}
b_\infty(0)
\\
=&
\phi_\infty  t_{(0)}e^{2\pi i\xi n}\cA(\alpha\mu^{-n})\cB^{-1}(\alpha\mu^{-n})
\cC(\alpha\mu_\infty\mu^{-n})\cB(\alpha\mu_\infty\mu^{-n}),
\end{split}
\end{equation*}

\begin{equation*}
\begin{split}
\bar p_{2n+1}=&-\phi_\infty
\bar t_{(0)} \bar t_\infty e^{-2\pi i\xi n}
\overline{\cA(\alpha\mu_\infty\mu^{-n})}\cB(\alpha\mu_\infty\mu^{-n})\\
&\times
\frac{b_0(\infty)K_\infty^{\alpha\mu_0^{-1}\mu^{-n}}(\infty)}{K_\infty^{\alpha\mu^{-n}}(\infty)}
\frac{K_\infty^{\alpha\mu^{-n}}(\infty)}{K_0^{\alpha\mu^{-n}}(0)}
\\
=&-\phi_\infty 
\bar t_{(0)} \bar t_\infty e^{-2\pi i\xi n}
\overline{\cA(\alpha\mu_\infty\mu^{-n})}\cB(\alpha\mu_\infty\mu^{-n})
\cC(\alpha\mu^{-n})\cB^{-1}(\alpha\mu^{-n}),
\end{split}
\end{equation*}
as well as  \eqref{modelco2}
$$
\bar \pi_{2n+1}=\bar t_{(0)}\bar t_{\infty} e^{-2\pi i\xi n}\phi_0
\frac{K_\infty^{\alpha\mu_0^{-1}\mu^{-n}}(0)}{K^{\alpha\mu_0^{-1}\mu^{-n}}_\infty(\infty)}
b_0(\infty)
\frac{K_\infty^{\alpha\mu_0^{-1}\mu^{-n}}(\infty)}{K^{\alpha\mu^{-n}}_\infty(\infty)}
\frac{K^{\alpha\mu^{-n}}_\infty(\infty)}{K^{\alpha\mu^{-n}}_0(0)}
$$
$$
=\bar t_{(0)}\bar t_{\infty} e^{-2\pi i\xi n}\phi_0
\overline{\cA(\alpha\mu_\infty\mu^{-(1+n)})}
\cB(\alpha\mu_\infty\mu^{-(1+n)})
\cC(\alpha\mu^{-n})\cB^{-1}(\alpha\mu^{-n}),
$$

$$
\sigma_{2n+1}=-\phi_0\cA(\alpha\mu^{-n})\cB^{-1}(\alpha\mu^{-n})
\overline{\cA(\alpha\mu_\infty\mu^{-(n+1)})}\cB(\alpha\mu_\infty\mu^{-(n+1)})
$$
$$
\bar\pi_{2n+2}=-t_{(0)}e^{2\pi i\xi(n+1)}\phi_0{\cA(\alpha\mu^{-n})}\cB^{-1}(\alpha\mu^{-n})
$$
$$
\times b_0(\infty)\frac{K_0^{\alpha\mu^{-(n+1)}}(0)}{K_0^{\alpha\mu_0^{-(n+1)}}(0)}
\frac{K_0^{\alpha\mu_0^{-1}\mu^{-n}}(0)}{K_\infty^{\alpha\mu_0^{-1}\mu^{-n}}(\infty)}
$$
$$
=-t_{(0)}e^{2\pi i\xi(n+1)}\phi_0{\cA(\alpha\mu^{-n})}\cB^{-1}(\alpha\mu^{-n})\cC(\alpha\mu_\infty\mu^{-(n+1)})
\cB(\alpha\mu_\infty\mu^{-(n+1)}).
$$

\end{proof}

\begin{remark}\label{rem43}
The structure of the reproducing kernels on the hyperelliptic Riemann surfaces is well known, see e.g. \cite{SY}. In particular, indeed $K_\infty^\alpha(0)=0$,  i.e., $\cA(\alpha)=0$, for some $\alpha$.
According to \eqref{modelco} and \eqref{modelco2}  it means that the corresponding $A=A(\alpha, t_{(0)})$ may degenerate, that is, $q_{2n}$ or $\sigma_{2n-1}$ vanishes for some $n$. 
Nevertheless all entries of $A$ and $A^{-1}$ have perfect sense. For example,
\begin{equation*}\label{nond}
\begin{split}
r_{2n+1}=&\frac{\bar p_{2n}\bar p_{2n+1}}{q_{2n}}\\
=&|\phi_\infty|
\cC(\alpha\mu^{-n})\cB(\alpha\mu^{-n})
\cC(\alpha\mu_\infty\mu^{-n})\cB^{-1}(\alpha\mu_{\infty}\mu^{-n}),
\\
-\rho_{2n}=&\frac{\bar \pi_{2n-1}\bar \pi_{2n}}{-\sigma_{2n-1}}\\
=&
|\phi_0|\cC(\alpha\mu^{-(n-1)})\cB(\alpha\mu^{-(n-1)})
\cC(\alpha\mu_\infty\mu^{-n})\cB^{-1}(\alpha\mu_{\infty}\mu^{-n}),
\end{split}
\end{equation*}
where
\begin{equation*}\label{inverse}
\begin{split}
A e_{2n-1}=&r_{2n-1}e_{2n-3}+\bar p_{2n-1}e_{2n-2}+q_{2n-1}e_{2n-1} +p_{2n}e_{2n}+
r_{2n+1}e_{2n+1},
\\
A^{-1} e_{2n}=&\rho_{2n}e_{2n-2}+\bar \pi_{2n}e_{2n-1}+\sigma_{2n}e_{2n} +
\pi_{2n+1}e_{2n+1}+\rho_{2n+2}e_{2n+2}.
\end{split}
\end{equation*}
\end{remark}

\bibliographystyle{amsplain}

  \vspace{.1in}

Research Institute for Symbolic Computation

Johannes Kepler University Linz

Altenbergerstr. 69, A-4040, Linz, Austria

E-mail: Ionela.Moale@risc.jku.at

 \vspace{.1in}

Abteilung f\"ur Dynamische Systeme und Approximationstheorie,

Johannes Kepler Universit\"at Linz,

A--4040 Linz, Austria

E-mail: Petro.Yudytskiy@jku.at
\end{document}